\documentclass{amsart}

\usepackage{amsmath,amssymb}

\usepackage[all]{xy}

\newtheorem{theorem}{Theorem}[section] 
\newtheorem{lemma}[theorem]{Lemma}

\theoremstyle{definition}
\newtheorem{definition}[theorem]{Definition}
\newtheorem{remark}[theorem]{Remark}
\newtheorem{claim}[theorem]{Claim}

\numberwithin{equation}{section}

\title[A conjecture on homotopy groups of spheres]
{A conjecture on homotopy groups of spheres, details on the algebra of higher cohomology operations}

\dedicatory{To Mamouka Jibladze on his fiftieth birthday} 

\author{Hans Joachim Baues}
 
\begin{document}

\begin{abstract}

The theory of secondary chomology operations leads to a conjecture concerning the algebra of higher cohomology operations in general. This conjecture is discussed here in detail and its connection with homotopy groups of spheres and the Adams spectral sequence is described.

\end{abstract}

\maketitle

\tableofcontents

\section{Bigraded differential algebras}
 Let $R$ be a commutative ring. A graded $R$--module $V = (V^n)_{n \in \mathbb Z}$ is a sequence of $R$--modules $V^n$. Let $\Sigma V = V[1]$ be the \emph{suspension} of $V$ which is the graded $R$--module satisfying $(\Sigma V)^n = V^{n-1}$. A bigraded $R$--module $W = (W_m^n)$ is a sequence of graded $R$--modules. For $x \in W_m^n$ we call $n = |x|$ the \emph{degree} of $x$ and we call $m = {\rm dim}(x)$ the \emph{dimension} of $x$. Let $W_m = (W_m^n)_{n \in \mathbb Z}$ be the graded submodule in dimension $m$. A bigraded chain complex $(W, d)$ is given by a differential
\begin{equation}
d_m = d: W_m \longrightarrow W_{m-1}
\end{equation}
with $|dx| = |x|$ and $dd = 0$. This is a graded object in the category of chain complexes. The homology ${\rm H}_m(W, d) = {\rm ker}d/{\rm im}d$ is a graded $R$--module.The \emph{$m$--truncation} of $(W, d)$ is the chain complex
\begin{equation}
{\rm tr}_m(W, d) = (\ \ldots \rightarrow 0 \rightarrow {\rm cok}d_{m+1} \rightarrow W_{m-1} \rightarrow W_{m-2} \rightarrow \ldots \ )
\end{equation}
which is concentrated in dimension $\leq m$. The \emph{$m$--cotruncation} of $(W, d)$ is the chain complex 
\begin{equation*}
{\rm cotr}_m(W, d) = (\ \ldots \rightarrow W_{m+2} \rightarrow W_{m+1} \rightarrow {\rm ker}d_{m}  \rightarrow 0 \rightarrow \ldots \ )
\end{equation*}
which is concentrated in dimension $\geq m$.

A bigraded algebra $A = (A_m^n)$ is a bigraded $R$--module together with a unit $1 \in A_0^0$ and an associative multiplication
\begin{equation}
\mu: A_m^n \otimes A_s^s \rightarrow A_{m+s}^{n+r} \quad {\rm with} \quad \mu(x \otimes y) = x \cdot y.
\end{equation}
We assume that $A$ is non--negatively graded, that is, $A_m^n = 0$ if $n < 0$ or $m < 0$. Let $\rm\bf A$ be the category of bigraded algebras and let $A \coprod B$ be the coproduct in $\rm\bf A$. Moreover, for a non--negatively bigraded set $E$ let ${\rm T}_R(E)$ be the \emph{free bigraded algebra} generated by $E$. Then ${\rm T}_R(E)$ is the free $R$--module generated by the free monoid ${\rm Mon}(E)$ consisting of all words $e_1 \ldots e_t$ with $e_1, \ldots, e_t \in E$ and $t \geq 0$.

A \emph{bigraded differential algebra} $(A, d)$ is a bigraded algebra $A$ which is also a bigraded chain complex satisfying
\begin{equation}
d(x \cdot y) = (dx) \cdot y + (-1)^{{\rm dim}(x)} x \cdot (dy).
\end{equation}
Then the $m$--truncation ${\rm tr}_m(A, d)$ is also a bigraded differential algebra while the $m$--cotruncation ${\rm cotr}_m(A, d)$ is an $(A, d)$--bimodule. The homology of $(A, d)$ is a bigraded algebra.

For a bigraded module $W$ we have the \emph{suspension} $\Sigma_s^r W$ with $(\Sigma_s^r W)_m^n = W_{m-s}^{n-r}$. Let $\Sigma_s^r: W \rightarrow \Sigma_s^r W$ be the map given by the identity. If $(W, d)$ is a bigraded chain complex, then $\Sigma_s^r (W, d)$ is also a chain complex with $d(\Sigma_s^r x) = (-1)^s \Sigma_s^r (dx)$.

\medskip

An element $x$ in a bigraded algebra $A$ is \emph{central} if, for all $y \in A$, one has $x \cdot y = (-1)^{{\rm dim}(x) \cdot {\rm dim}(y)} y \cdot x$.

A (bigraded differential) algebra $A = (A, d)$ has a \emph{$\Sigma$--structure} if an element $[1] \in A_1^1$ is given with $d[1] = 0$ such that $[1]$ is central in ${\rm H}_{\ast}(A)$ and the chain map
\begin{equation}
\Sigma_1^1(A, d) \stackrel{[1]\cdot}{\longrightarrow} {\rm cotr}_1(A, d) 
\end{equation}
which carries $\Sigma_1^1 x$ to $[1] \cdot x$ induces isomorphisms in homology. This implies that for $m \geq 0$ we have isomorphisms of ${\rm H}_0(A, d)$--bimodules
\begin{equation*}
{\rm H}_m(A, d) = \Sigma^m {\rm H}_0(A, d).
\end{equation*}
In a similar way we define the $\Sigma$--structure of a left $A$--module where $A$ is an algebra with $\Sigma$--structure.
\begin{claim}
Let $A$ be an algebra with $\Sigma$--structure and let $X$ and $Y$ be left $A$--modules with $\Sigma$--structure. Then the bigraded $R$--module
\begin{equation*}
{\rm Ext}_A(X, Y)
\end{equation*}
is defined in terms of a``resolution" of $X$ in the category of left $A$--modules. This generalizes the secondary $\rm Ext$--groups studied in \cite{BauesJibladze1}.
\end{claim}

\medskip

Moreover, we need the following notion of truncated algebras. An \emph{$(m)$--algebra} $B$ is a bigraded differential algebra with $B_k = 0$ for $k \geq m$, or, equivalently,
\begin{equation}\label{malgebra}
{\rm tr}_{m-1}(B) = B.
\end{equation}
For example, the truncation ${\rm tr}_{m-1}(A)$ of a bigraded differential algebra $A$ is an $(m)$--algebra. We say that an $(m)$--algebra $B$ with $m \geq 2$ has a \emph{$\Sigma$--structure} if an element $[1] \in B_1^1$ is given with $d[1] = 0$ such that $[1]$ is central in the homology algebra ${\rm H}_{\ast}(B)$ and the chain map
\begin{equation}
\Sigma_1^1{\rm tr}_{m-2}(B) \stackrel{[1]\cdot}{\longrightarrow} {\rm cotr}_1(B) 
\end{equation}
induces an isomorphism in homology. This implies that one has an isomorphism of ${\rm H}_0(B)$--bimodules
\begin{equation*}
{\rm H}_k (B) = \Sigma^k {\rm H}_0(B) \quad {\rm for} \quad 0 \leq k \leq m-1.
\end{equation*}
A $(2)$--algebra is the same as a \emph{pair algebra} considered in \cite{Baues1} and a $\Sigma$--structure of a pair algebra is described in 1.2 of \cite{BauesJibladze1}. We also consider, for an $(m)$--algebra $B$, the left $B$--modules $X$ which are bigraded chain complexes with $X_k = 0$, for $k \geq m$, and for which $B$ acts from the left on $X$. Then, if $B$ has a $\Sigma$--structure, there is also a $\Sigma$--structure defined for $X$. 
\begin{claim}
Let $m \geq 1$ and let $B$ be an $(m)$--algebra with $\Sigma$--structure and let $X, Y$ be left $B$--modules with a $\Sigma$--structure. Then the graded module
\begin{equation*}
{\rm Ext}_B(X,Y)
\end{equation*}
\end{claim}
is defined in terms of a ``resolution'' of $X$ in the category of left $B$--modules. In fact, for $m = 2$, this $\rm Ext$--module coincides with the secondary derived functor studied in \cite{BauesJibladze1}, compare chapter 1 in \cite{BauesJibladze2}. For $m = 1$, a $(1)$--algebra is the same as a graded algebra and in this case ${\rm Ext}_B(X,Y)$ coincides with the classical derived functor.

\medskip

It is clear that the $(m-1)$--truncation ${\rm tr}_{m-1}(A)$ of a bigraded differential algebra $A$ is an $m$--algebra and that ${\rm tr}_{m-1}(A)$ has a $\Sigma$--structure if $A$ has one. A similar statement holds for the truncation of $A$--modules.

\begin{claim}\label{SS}
Let $A$ be an algebra with $\Sigma$--structure and let $X, Y$ be $A$--modules with $\Sigma$--structure. Then, for $m \geq 1$, one obtains the $(m)$--algebra $B^{(m)} = {\rm tr}_{m-1}(A)$ and the $B^{(m)}$--modules $X^{(m)} =  {\rm tr}_{m-1}(X)$ and $Y^{(m)} =  {\rm tr}_{m-1}(Y)$, where $B^{(m)}, X^{(m)}$ and $Y^{(m)}$ have $\Sigma$--structures. Hence the graded $\rm Ext$--modules, $m \geq 1$,
\begin{equation*}
{\rm E}_{m+1} = {\rm Ext}_{B^{(m)}}(X^{(m)}, Y^{(m)})
\end{equation*}
are defined. Moreover, $({\rm E}_2, {\rm E}_3, {\rm E}_4, \ldots \ )$ form a spectral sequence which converges to the bigraded module ${\rm Ext}_A(X, Y)$.
\end{claim}

\section{The algebra of higher cohomology operations}\label{highcohom}
Let $p$ be a prime number, let $\mathbb F = \mathbb Z/p$ be the prime field and let $\mathbb G = \mathbb Z/p^2$ be the quotient ring. Then there is a ring homomorphism $\mathbb G \rightarrow \mathbb F$ and a canonical long exact sequence
\begin{equation}\label{Gchain}
0 \longleftarrow \mathbb F \longleftarrow \mathbb G \stackrel{d}{\longleftarrow} \mathbb G \stackrel{d}{\longleftarrow} \mathbb G \stackrel{d}{\longleftarrow} \ldots
\end{equation}
where each differential $d$ is given by multiplication with $p$.
\begin{definition}
Let $\mathbb G_{\ast}$ be the bigraded algebra over the ring $\mathbb G$ generated by elements $[1]$ and $[p]_s, s \geq 1$, where ${\rm dim}[1] = {\rm deg}[1] = 1, {\rm dim}[p]_s = 0$ and ${\rm deg}[p]_s = s$. The relations for $\mathbb G_{\ast}$ are
\begin{eqnarray*}
&&[p]_s \cdot [p]_t = 0 \quad {\rm for} \quad s, t \geq 1,\\
&&[1] \cdot [p]_s = [p]_s \cdot [1] \quad {\rm for} \quad s \geq 1.
\end{eqnarray*}
Hence $\mathbb G_{\ast}$ is free as a $\mathbb G$--module generated by the basis elements $[1]^r \cdot [p]_s$, for $r, s \geq 0$, where $[1]^r$ is the $r$--th power with $[1]^0 = [p]_0 = 1$, the unit of $\mathbb G_{\ast}$.  We define a differential $d$ of $\mathbb G_{\ast}$ by
\begin{eqnarray*}
&& d[1] = 0 \\
&& d [p]_s = p \cdot [p]_{s-1} \quad {\rm for} \quad s \geq 1.
\end{eqnarray*}
Hence $\mathbb G_0$ coincides with the chain complex (\ref{Gchain}).
\end{definition}
\begin{lemma}
The element $[1]$ is central in $\mathbb G_{\ast}$ and the differential algebra $\mathbb G_{\ast}$ has a $\Sigma$--structure with homology ${\rm H}_n (\mathbb G_{\ast}) = \Sigma^n \mathbb F, n \geq 0$.
\end{lemma}
Let $\mathbb G_{\ast}^{(m)} = {\rm tr}_{m-1}(\mathbb G_{\ast})$ for $m \geq 1$. Then $\mathbb G_{\ast}^{(1)} = \mathbb F$ and, for $m = 2$, the truncation
\begin{equation}
G_{\ast}^{(2)} = (\  \Sigma \mathbb F \oplus \mathbb F \stackrel{d}{\longrightarrow} \mathbb G \ )
\end{equation}
coincides with the pair algebra $\mathbb G^{\Sigma}$ in 12.1.3 of \cite{Baues1}.

Let $\mathcal A$ be the Steenrod algebra over $\mathbb F$ which is generated by a canonical graded set $E_A$ given by
\begin{equation*}
E_A = 
\begin{cases}
\{ Sq^i \ | \ i \geq 1 \} & {\rm for} \  p = 2,\\
\{\beta \} \cup \{ P^i , P^i_{\beta} \ | \ i \geq 1 \} & {\rm for} \  p \ odd.
\end{cases}
\end{equation*}
Compare 5.5.1 in \cite{Baues1}. We consider a bigraded set $E$ with $E_d^k$ empty for $d < 0$ or $k < 0$ or $k \leq d$. Moreover,
\begin{eqnarray}
E_0 & = & E_A \\
E_1 & = & {\rm graded \ set \ of \ Adem \ relations} \\
E_2 & = & {\rm graded \ set \ of \ relations \ among \ relations}
\end{eqnarray}
Here $E_2$ is a set of generators of the bimodule $K_A$ in 5.5.3 (3) in \cite{Baues1}. We consider a bigraded differential algebra
\begin{equation}\label{B}
B_{\ast} = (\mathbb G_{\ast} \coprod T_{\mathbb G}(E), d)
\end{equation}
with inclusion $\iota$ and augmentation $\varepsilon$,
\begin{eqnarray*}
\mathbb G_{\ast} \stackrel{\iota}{\longrightarrow} B_{\ast} \stackrel{\varepsilon}{\longrightarrow} \mathbb G_{\ast}.
\end{eqnarray*}
Here $\varepsilon$ is the identity on $\mathbb G_{\ast}$ and carries $E$ to $0$. The maps $\iota$ and $\varepsilon$ are maps of differential algebras, so that $\mathbb G_{\ast}$ is a left $B_{\ast}$--module via $\varepsilon$. Using $[1] \in \mathbb G_{\ast}$, the algebra $B_{\ast}$ has a $\Sigma$--structure with
\begin{equation*}
{\rm H}_n B_{\ast} = \Sigma^n \mathcal A, \quad n \geq 1,
\end{equation*} 
as graded $\mathcal A$--bimodules, where $\mathcal A$ coincides with the graded algebra ${\rm H}_0 B_{\ast}$. Morevover, the differential $d = d_n: B_n \rightarrow B_{n-1}$ induces a split morphism $d_n: {\rm cok}(d_{n+1}) \twoheadrightarrow {\rm image}(d_n)$ of $\mathbb G$--modules.

\medskip

\begin{claim}
There exists an algebra $B_{\ast}$ with the properties in (\ref{B}), such that the truncation
\begin{equation*}
B_{\ast}^{(2)} = {\rm tr}_1 (B_{\ast})
\end{equation*}
coincides with the pair algebra of secondary cohomology operations computed in \cite{Baues1}. Moreover, let $\pi_{\ast}^S$ be the algebra of stable $p$--local homotopy groups of spheres with the Adams filtration and let ${\rm Gr}(\pi_{\ast}^S)$ be the associated bigraded algebra. Then there is an isomorphism of bigraded algebras
\begin{equation*}
{\rm Gr}(\pi_{\ast}^S) = {\rm Ext}_{B_{\ast}}(\mathbb G_{\ast}, \mathbb G_{\ast}).
\end{equation*}
In fact, the spectral sequence of Claim \ref{SS} coincides with the Adams spectral sequence $(E_2, E_3, \ldots)$ which converges to ${\rm Gr}(\pi_{\ast}^S)$. Here we have
\begin{eqnarray*}
&&E_2 = {\rm Ext}_{\mathcal A}(\mathbb F, \mathbb F) \quad {\rm and}\\
&&E_{m+1} = {\rm Ext}_{B_{\ast}^{(m)}}(\mathbb G_{\ast}^{(m)}, \mathbb G_{\ast}^{(m)}), \quad m \geq 1.
\end{eqnarray*}
For $m = 2$, this equation is proved in \cite{BauesJibladze1}.
\end{claim}

\medskip

We shall construct the algebra $B_{\ast}$ by the inductive definition of the truncation $B_{\ast}^{(m)}, m \geq 1$. For $m = 1$, we have $B_{\ast}^{(1)} = \mathcal A$ and for $m = 2$, we get the pair algebra $B_{\ast}^{(2)}$ in \cite{Baues1}. Hence we have to define $B_{\ast}^{(m)}$ for $m \geq 3$. We do this by the ``strictification'' of the higher Steenrod algebra.

\section{The higher Steenrod algebra}\label{steenrod}
We define the higher Steenrod algebra $[[\mathcal A]]_{\ast}$ which generalizes the secondary Steenrod algebra $[[\mathcal A]]$ in 2.3 \cite{Baues1}. 
 
Let ${\rm\bf Top}^{\ast}$ be the category of pointed topological spaces and let $[[{\rm\bf Top}^{\ast}]]$ be the associated groupoid enriched category termed track category. A \emph{track} $H: f \Rightarrow g$ of maps $f, g: X \rightarrow Y$ in ${\rm\bf Top}^{\ast}$ is a homotopy class of pointed homotopies $f \simeq g$. Let $[[X, Y]]$ be the groupoid of maps and such tracks and let $[X, Y] = \pi_0[[X, Y]]$ be the set of homotopy classes of maps $X \rightarrow Y$. For tracks $H: f \Rightarrow g$ and $G: h \Rightarrow f$ let $H \square G: h \Rightarrow g$ be the composition of tracks and let $0_f^{\square}: f \Rightarrow f$ be the identity track. We have the corresponding composition $H \square G$ of homotopies which, however, is not associative. Therefore we also use \emph{Moore homotopies} for which the composition $H \square G$ is associative; compare the corresponding notion of Moore loop spaces in the literature.

Let $X \wedge Y$ be the smash product of pointed spaces and, for maps $\alpha: P \wedge X \rightarrow Y$ and $\beta: Q \wedge Z \rightarrow X$, let $\alpha \circ \beta$ be the composite
\begin{equation}
\alpha \circ \beta: P \wedge Q \wedge Z \stackrel{P \wedge \beta}{\longrightarrow} P \wedge X \stackrel{\alpha}{\longrightarrow} Y.
\end{equation}
This pairing is associative. For tracks $H: \alpha \Rightarrow \alpha'$ and $G: \beta \Rightarrow \beta'$, we then have a similar pasting operation $H \ast G: \alpha \beta \Rightarrow \alpha' \beta'$.

Recall that the Eilenberg--{Mac\,Lane} space
\begin{equation*}
Z^n = K(\mathbb F,n)
\end{equation*}
is a topological $\mathbb F$--vector space with the properties described in 2.1 \cite{Baues1}. We fix a homotopy equivalence
\begin{equation}\label{overlinern}
r_n: Z^n \stackrel{\sim}{\longrightarrow} \Omega Z^{n+1}
\end{equation} 
which is $\mathbb F$--linear, see 2.1.7 \cite{Baues1}, hence the map $r_n$ is adjoint to the composite
\begin{equation*}
\overline r_n: Z^n \wedge S^1 \longrightarrow Z^n \wedge Z^1 \stackrel{\mu}{\longrightarrow} Z^{n+1}.
\end{equation*} 

For a pointed space $Q$ let $r_n: Q \wedge Z^n \stackrel{\sim}{\rightarrow} \Omega(Q \wedge Z^{n+1})$ be adjoint to the map $Q \wedge \overline r_n$.
\begin{definition}
Let $[[\mathcal A^k]]^Q$ be the following groupoid, $k \geq 1$. Objects $(\alpha, H_{\alpha})$ in $[[\mathcal A^k]]^Q$ are sequences of maps in ${\rm\bf Top}^{\ast}$
\begin{equation*}
\alpha = ( \alpha_n: Q \wedge Z^n \longrightarrow Z^{n+k})_{n \in \mathbb N}
\end{equation*}
together with sequences of Moore homotopies $H_{\alpha} = (H_{\alpha,n})_{n \in \mathbb N}$ for the diagram
\begin{equation*}
\xymatrix{
Q \wedge Z^n \ar[rr]^{\alpha_n} \ar[dd]_{r_n} && Z^{n+k} \ar[dd]^{r_{n+k}} \\
& \Rightarrow & \\
\Omega(Q \wedge Z^{n+1}) \ar[rr]^{\Omega \alpha_{n+1}} && \Omega Z^{n+k+1},}
\end{equation*}
that is, $H_{\alpha,n}: \Omega \alpha_{n+1} r_n \Rightarrow r_{n+k} \alpha_n$. For $k \leq 0$, let $[[\mathcal A^k]]^Q = 0$ be the trivial groupoid if $Q$ is path connected. If $Q = S^0$ is the zero--sphere, then let $[[\mathcal A^0]]^{S^0} = \mathbb F$ be the discrete groupoid given by $\mathbb F$ and let $[[\mathcal A^k]]^{S^0} = 0$ for $k < 0$.

\medskip

We call the object $(\alpha, H_{\alpha})$ \emph{strict} if $\Omega \alpha_{n+1} r_n = r_{n+k} \alpha_n$ and $H_{\alpha}$ is the identity homotopy.
\end{definition} 
 
\medskip

For $k > 0$, we define morphisms $H: (\alpha, H_{\alpha}) \Rightarrow (\beta, H_{\beta})$ in the groupoid $[[\mathcal A^k]]^Q$ by sequences of tracks
\begin{equation*}
H = (H_n: \alpha_n \Rightarrow \beta_n)_{n \in \mathbb Z}
\end{equation*}
in $[[{\rm\bf Top}^{\ast}]]$, for which the pasting of tracks in the following diagram coincides with $H_{\beta,n}$. 
\begin{equation*}
\xymatrix{
Q \wedge Z^n \ar@{=}[dd] \ar[rr]^{\beta_n} && Z^{n+k} \ar@{=}[dd]\\
& \Uparrow H_n & \\
Q \wedge Z^n \ar[dd] _{r_n} \ar[rr]^{\alpha_n} && Z^{n+k} \ar[dd]^{r_{n+k}} \\
& \stackrel{H_{\alpha,n}}{\Rightarrow} & \\
\Omega(Q \wedge Z^{n+1}) \ar@{=}[dd] \ar[rr]^{\Omega \alpha_{n+1}} && \Omega Z^{n+k+1} \ar@{=}[dd] \\ & \Downarrow \Omega H_{n+1} & \\
\Omega(Q \wedge Z^{n+1}) \ar[rr]_{\Omega \beta{n+1}} && \Omega Z^{n+k+1} }
\end{equation*}  
That is, the following equation holds in $[[{\rm\bf Top}^{\ast}]]$,
\begin{equation*}
H_{\beta,n} = (\Omega H_{n+1}) r_n \square H_{\alpha,n} \square r_{n+k} H_n^{\rm op}.
\end{equation*}
Composition in $[[\mathcal A^k]]^Q$ is defined by $(H \square G)_n = H_n \square G_n$. One readily checks that $[[\mathcal A^k]]^Q$ is a well-defined groupoid. Moreover, there is a composition functor between groupoids
\begin{equation}\label{comp}
[[\mathcal A^r]]^P \times [[\mathcal A^k]]^Q \stackrel{\circ}{\longrightarrow} [[\mathcal A^{k+r}]]^{P \wedge Q}
\end{equation}
which is defined on objects by
\begin{equation*}
(\alpha', H_{\alpha'}) \circ (\alpha, H_{\alpha}) = (\alpha'_{n+k} \circ \alpha_n, H_{\alpha', n+k} \ast H_{\alpha, n})_{n \in \mathbb Z},
\end{equation*}
where $\ast$ is the pasting operation. Moreover, on morphisms $H: (\alpha, H_{\alpha}) \Rightarrow (\beta, H_{\beta})$ and $H': (\alpha', H_{\alpha'}) \Rightarrow (\beta', H_{\beta'})$, the composition functor is defined by
\begin{equation*}
H \circ H' = (H_n' \ast H_n: \alpha_n' \circ \alpha_n \Rightarrow \beta_n' \circ \beta_n)_{n \in \mathbb Z}.
\end{equation*}
\begin{lemma}
The groupoid $[[\mathcal A^k]]^Q$ is an $\mathbb F$--vector space object in the category of groupoids.
\end{lemma}
\begin{proof}
This follows since $Z^n$ is a topological $\mathbb F$--vector space and the maps $r_n$ for $Z^n$ are $\mathbb F$--linear.
\end{proof}
 
A map $f: P \rightarrow Q$ induces a functor between groupoids
\begin{equation*}
f^{\ast}: [[\mathcal A^k]]^Q \longrightarrow [[\mathcal A^k]]^P
\end{equation*} 
which carries $(\alpha, H_{\alpha})$ to $(\alpha(f \wedge 1), H_{\alpha}(f \wedge 1))$ and $H: (\alpha, H_{\alpha}) \Rightarrow (\beta, H_{\beta})$ to $f^{\ast} H$ with $(f^{\ast} H)_n = H(f \wedge 1)$. 

\begin{remark}
Let $H \mathbb F$ be the Eilenberg--{Mac\,Lane} spectrum. Then the spectrum $Q \wedge H \mathbb F$ is defined and the graded set of homotopy classes of maps $Q \wedge H \mathbb F \rightarrow H \mathbb F$ in the cagetory of spectra coincides with $(\pi_0[[\mathcal A^k]]^Q)_{k \in \mathbb Z}$. In particular, for $Q = S^0$, we obtain $\pi_0[[\mathcal A^k]]^{S^0} = \mathcal A^k$, where $\mathcal A$ is the Steenrod algebra.
\end{remark} 
 
\medskip

There is an equivalence of graded groupoids
\begin{equation*}
[[\mathcal A]]^{S^0} \stackrel{\sim}{\longrightarrow} [[\mathcal A]],
\end{equation*} 
where $[[\mathcal A]]$ is the secondary Steenrod algebra defined in 2.5. \cite{Baues1}. In fact, $[[\mathcal A]]$ is defined in the same way as $[[\mathcal A]]^{S^0}$, except that the Moore homotopies $H_{\alpha, n}$ above correspond to the tracks $H_{\alpha,n}$ in the definition of $[[\mathcal A]]$.

\begin{definition}
Let $S^d$ be the sphere of dimension $d$. The \emph{higher Steenrod algebra} is the bigraded groupoid
\begin{equation*}
[[\mathcal A]]_{\ast} = \big( [[\mathcal A^k]]^{S^d}\big)_{k, d \in \mathbb Z}
\end{equation*}
Here $k$ is the \emph{degree} and $d$ is the \emph{dimension}. For $k < 0$ or $d < 0$, let $ [[\mathcal A^k]]^{S^d} = 0$ be the trivial groupoid. We have the associative pairing between groupoids (since $S^d \wedge S^{\ell} = S^{d+\ell}$)
\begin{equation*}
[[\mathcal A^r]]^{S^d} \times  [[\mathcal A^k]]^{S^{\ell}} \stackrel{\circ}{\longrightarrow}  [[\mathcal A^{k+r}]]^{S^{d+\ell}}.
\end{equation*}
\end{definition} 
Here $[[\mathcal A^r]]^{S^d}$ is an $\mathbb F$--vector space object in the catgeory of groupoids. The pairing, however, is not bilinear, but linear on the left hand side, that is, $(\alpha + \alpha') \circ \beta = \alpha \circ \beta + \alpha' \circ \beta$. Hence $[[\mathcal A]]_{\ast}$ is a monoid in the category of bigraded groupoids and, due to the $\mathbb F$--vector space structure, a \emph{near $\mathbb F$--algebra}, not an $\mathbb F$--algebra object. 

\medskip

The secondary Steenrod algebra $[[\mathcal A]]$ in 2.5 \cite{Baues1} corresponds to the dimension $0$ part of the higher Steenrod algebra $[[\mathcal A]]_{\ast}$, that is, $[[\mathcal A]] \sim [[\mathcal A]]^{S^0}$. This indicates how to generalize the methods in the book \cite{Baues1} for the study of $[[\mathcal A]]_{\ast}$.
 
Let $\Sigma^k \mathcal A$ be the $k$--fold suspension of the Steenrod algebra $\mathcal A$ which is an $\mathcal A$--bimodule.
\begin{lemma}
\begin{equation*}
\pi_0\big( [[\mathcal A]]^{S^d} \big) = \Sigma^d \mathcal A.
\end{equation*}
\end{lemma} 
\begin{proof}
The maps $S^d \wedge H \mathbb F \rightarrow H \mathbb F$ between spectra correspond to the elements in $\Sigma^d \mathcal A$.
\end{proof}

For a groupoid $[[G]]$, let $[[G]]_0$ be the set of objects and let $[[G]]_1$ be the set of morphisms in $[[G]]$. Moreover, if $[[G]]$ is pointed by $0$, then let $[[G]]_{01}$ be the set of morphisms $H: f \Rightarrow 0$, see 5.1.1 \cite{Baues1}. Let $\partial: [[G]]_{01} \rightarrow [[G]]_0$ be defined by $\partial H = f$.
\begin{lemma}\label{exseq}
There is an exact sequence of graded $\mathbb F$--vector spaces ($d \geq 0$)
\begin{equation*}
0 \longrightarrow \Sigma^{d+1} \mathcal A \longrightarrow [[\mathcal A]]_{01}^{S^d} \stackrel{\partial}{\longrightarrow} [[\mathcal A]]_0^{S^d} \longrightarrow \Sigma^d \mathcal A \longrightarrow 0
\end{equation*}
Moreover, there is a strict element $[1] \in [[\mathcal A^1]]_0^{S^1}$ mapping to $\Sigma 1 \in \Sigma \mathcal A$.
\end{lemma} 
\begin{proof}
The element $[1]$ is given by the composite
\begin{equation*}
\alpha_n: S^1 \wedge Z^n \stackrel{T}{\longrightarrow} Z^n \wedge S^1 \stackrel{\overline r_n}{\longrightarrow} Z^{n+1} \stackrel{\tau_n}{\longrightarrow} Z^{n+1},
\end{equation*}
where $T$ is the interchange map and $\overline r_n$ is defined in (\ref{overlinern}). Moreover, $\tau_n \in \sigma_{n+1}$ is the permutation of $\{1, \ldots, n\}$ and $(n+1)$. Then $H_{\alpha,0} = 0^{\square}$ is the trivial homotopy. One readily checks that the map $[1]$ is well--defined.
\end{proof}

\medskip

Let $D^{d+1}$ be the $(d+1)$--ball (or $(d+1)$--disc) given by the reduced cone of the sphere $S^d$. We have the inclusion $\iota: S^d \subset D^{d+1}$ of the boundary and the quotient map $q: D^{d+1} \rightarrow D^{d+1} / S^d = S^{d+1}$. 
\begin{lemma}\label{exseq2}
The maps $\iota$ and $q$ induce an exact sequence of $\mathbb F$--vector spaces
\begin{equation*}
[[\mathcal A^k]]_0^{S^{d+1}} \stackrel{q^{\ast}}{\longrightarrow} [[\mathcal A^k]]_0^{D^{d+1}} \stackrel{\partial = \iota^{\ast}}{\longrightarrow} [[\mathcal A^k]]_0^{S^d}.
\end{equation*}
Moreover, there is a surjective $\mathbb F$--linear map
\begin{equation*}
[[\mathcal A^k]]_0^{D^{d+1}} \stackrel{\pi}{\twoheadrightarrow}  [[\mathcal A^k]]_{01}^{S^d} 
\end{equation*}
which carries a map $H$, considered as a homotopy $H: \alpha \rightarrow 0$ with $\alpha = \partial H$, to the associated track.
\end{lemma} 

Let $I = [0,1]$ be the \emph{unit interval} and, for $d \geq 0$, let $I^d = I \wedge \ldots \wedge I$ be the $d$--fold smash product of $I$ which is the $0$--sphere for $d = 0$. For $d > 0$, we fix a homeomorphism $\lambda: I^d \rightarrow D^d$, so that $\lambda$ induces an isomorphism of $\mathbb F$--vector spaces
\begin{equation}
\lambda^{\ast}: [[\mathcal A^k]]^{D^d} \cong [[\mathcal A^k]]^{I^d}.
\end{equation}
Since $I^d \wedge I^{\ell} = I^{d+\ell}$, we have the associative multiplication
\begin{equation}\label{multiplication}
[[\mathcal A^r]]^{I^d} \times [[\mathcal A^k]]^{I^{\ell}} \stackrel{\circ}{\longrightarrow}  [[\mathcal A^{k+r}]]^{I^{d+\ell}}
\end{equation}
which is $\mathbb F$--linear on the left hand side.

\section{The inductive definition of $B_{\ast}$}
The algebra $B_{\ast}$ of higher cohomology operations  with the properties in Section \ref{highcohom} is defined inductively as a strictification of the higher Steenrod algebra $[[\mathcal A]]_{\ast}$ in Section \ref{steenrod}. As in (\ref{B}), we have
\begin{eqnarray*}
&& B_{\ast} =  (\mathbb G_{\ast} \coprod T_{\mathbb G}(E), d) \\
&& B_{\ast}^{(m)} =  {\rm tr}_{m-1}(B_{\ast}), m \geq 1.
\end{eqnarray*}
We define inductively $B_{\ast}^{(m)}$ and the $m$--dimensional part $E_m$ of the bigraded set $E, m \geq 0$.

\medskip

We observe that $B_0 = T_{\mathbb G}(E_0)$ and that $B_{\ast}$ is free as a $\mathbb G$--module with basis elements given by the words
\begin{equation}\label{basisel}
\alpha_0 g_1 \alpha_1 g_2 \alpha_2 \ldots g_k \alpha_k,
\end{equation}
where $k \geq 0, \alpha_i \in {\rm Mon}(E)$, with $\alpha_1, \ldots, \alpha_{k-1}$ different from $1$ and $g_i \in \{ [1]^n[p]_t, s, t \geq 0, s+t > 0\}$.It is clear that $B_s$ depends only on the elements $[1]^s[p]_t$ and on the sets $E_0, E_1, \ldots, E_s$.

The $(m)$--algebra $B_{\ast}^{(m)}$ will be constructed as a pull back diagram of $\mathbb G$--modules, $m \geq 2$,
\begin{equation}\label{pb}
\xymatrix{
B_{m-1}^{(m)} \ar[d]_d \ar[r]^{\overline s_{m-1}} & [[\mathcal A]]_{01}^{S^{m-2}} \ar[d]^{\partial}\\
{\rm ker}(d_{m-2}) \ar[r]^{s_{m-2}} & [[\mathcal A]]_{0}^{S^{m-2}}}
\end{equation}
Here $B_{\ast}^{(m)}$ as a chain complex has the form
\begin{equation*}
\xymatrix{
0 \ar[r] & B_{m-1}^{(m)} \ar[r]^d & B_{m-2} \ar[r]^{d_{m-2}} & B_{m-3} \ar[r] & \ldots \ar[r] & B_1 \ar[r]^{d_1} & B_0 \ar[r] &0}
\end{equation*}
Since $s_{m-1}$ surjects onto ${\rm cok}(\partial)= \Sigma^{m-2} \mathcal A$, we see that the pull back diagram is also a push out diagram. Therefore we obtain, by (\ref{exseq}), the exact sequence
\begin{equation*}
\xymatrix{
0 \ar[r] & \Sigma^{m-1} \mathcal A \ar[r]^{\iota} & B_{m-1}^{(m)} \ar[r]^d & {\rm ker}(d_{m-2}) \ar[r]^q & \Sigma^{m-2} \mathcal A \ar[r]  & 0}
\end{equation*}
An element in $B_{m-1}^{(m)}$ is a pair $(H, x)$ with $x \in {\rm ker}(d_{m-2})$ and $H: s_{m-2} \Rightarrow 0$ in the groupoid $[[\mathcal A]]^{S^{m-2}}$, so that $d(H, x) = x$. We have canonical elements
\begin{equation}
e^t_{m-1} \in B_{m-1}^{(m)}, t \geq 0,
\end{equation}
with $e^0_{m-1} = \iota(\Sigma^{m-1} 1)$ and, for $t > 0$,
\begin{equation*}
\begin{cases}
d(e^t_{m-1}) = p (-1)^{m-1} e^{t-1}_{m-1}, \\
\overline s_{m-1}(e^t_{m-1}) = 0.
\end{cases}
\end{equation*}

We call a graded set $E_{m-1}$ a \emph{generating set} for $B_{m-1}^{(m)}$ if there is a function
\begin{equation}\label{genset}
e'_{m-1}: E_{m-1} \rightarrow B_{m-1}^{(m)}
\end{equation}
with the following properties. Let
\begin{equation*}
e = e_{m-1}: \mathbb G_{\ast} \coprod T(E_0, \ldots, E_{m-1}) \rightarrow B_{\ast}^{(m)}
\end{equation*}
be the algebra map which is the identity in dimension $< m-1$ and, in dimension $m-1$, is given by
\begin{eqnarray*}
&& e\big( [1]^{m-1} [p]_t \big) = e^t_{m-1}, \quad t \geq 0, \\
&& e(x) = e'_{m-1}(x) \quad {\rm for} \ x \in E_{m-1}.
\end{eqnarray*}
We say that $e$ is \emph{generating} if $e$ is surjective in dimension $m-1$. Now assume $E_0, \ldots, E_{m-2}$ are given and the $m$--algebra $B_{\ast}^{(m)}$ is constructed with the properties in (\ref{pb}). Then we choose a generating set $E_{m-1}$ of $B_{m-1}^{(m)}$ as in (\ref{genset}) and a lift $s_{m-1}^E$ as in the commutative diagram
\begin{equation}
\xymatrix{
E_{m-1} \ar@{-->}[dd]_{s_{m-1}^E} \ar[rr]^{e'_{m-1}} && B_{m-1}^{(m)}  \ar[d]^{\overline s_{m-1}} \\
&& [[\mathcal A]]_{01}^{S^{m-2}}\\
[[\mathcal A_0]]^{I^{m-1}} && [[\mathcal A]]_0^{D^{m-1}}  \ar[ll]^{\cong}_{\lambda^{\ast}} \ar@{->>}[u]_{\pi}}
\end{equation}
Accordingly, there are functions $s_k^E$ for $k \leq m-1$ by induction. By the multiplication (\ref{multiplication}), the functions $s_0^E \ldots s_{m-1}^E$ define a unique monoid map
\begin{equation*}
s_{\ast}^{\rm Mon}: {\rm Mon}(E_0, \ldots, E_{m-1}) \rightarrow \big( [[\mathcal A]]_0^{I^k}\big)_{k \geq 0}.
\end{equation*}

We now define
\begin{equation}
B_{m-1} = \big(\mathbb G_{\ast} \coprod T(E_0, \ldots, E_{m-1})\big)_{m-1}
\end{equation}
and we define $d_{m-1}$ by the composite
\begin{equation*}
d_{m-1}: B_{m-1} \stackrel{e}{\longrightarrow} B_{m-1}^{(m)} \stackrel{d}{\longrightarrow} B_{m-2},
\end{equation*}
where we use $e$ in (\ref{genset}). Moreover, we define
\begin{equation}
s_{m-1}^B: B_{m-1} \longrightarrow [[\mathcal A]]_0^{I^{m-1}}
\end{equation}
by the $\mathbb G$--linear map which, on basis elements (\ref{basisel}), is given by products of
\begin{equation*}
\begin{array}{lllll}
& s_{m-1}^B(\alpha_i) & = & s_{\ast}^{\rm Mon}(\alpha_i) & {\rm for} \ \alpha_i \in {\rm Mon}(E_0, \ldots, E_{m-1}),\\
& s_{m-1}^B([p]_t) & = & 0 & {\rm for} \ t > 0, \\
& s_{m-1}^B([1]^s) & = & \lambda^{\ast} q^{\ast}([1]^s), &
\end{array}
\end{equation*}
where $[1]^s \in [[\mathcal A]]_0^{S^s}$ is given by the $s$--fold product of the strict element $[1]$ in (\ref{exseq}). Here we use $q^{\ast}$ in Lemma \ref{exseq2} and $\lambda^{\ast}$ in (\ref{multiplication}). However, the map $s_{\ast}^B$ is not multiplicative, that is, for $x, y \in B_{\ast}$ with $x \cdot y \in B_{m-1}$, the element $s_{\ast}^B(x \cdot y)$ does not coincide with the product $s_{\ast}^B(x) \cdot s_{\ast}^B(y)$. We obtain the following diagram extending (\ref{pb}).
\begin{equation}\label{extdiag}
\xymatrix{
{\rm ker} (d_{m-1}) \ar@{>->}[d] \ar[rrr]^{s_{m-1}} &&& [[\mathcal A]]_0^{S^{m-1}} \ar@{>->}[d]\\
B_{m-1} \ar@/^1.5pc/[rrr]^{(\lambda^{\ast})^{-1}s^B_{m-1}} \ar[r] \ar[d]_{d_{m-1}} & B_{m-1}^{(m)} \ar[r] \ar[dl]^d & [[\mathcal A]]_0^{S^{m-2}} \ar[dr]_{\partial} & [[\mathcal A]]_0^{D^{m-1}} \ar[l] \ar[d]^{\iota_{\ast}} \\
{\rm ker}(d_{m-2}) \ar[rrr]_{s_{m-2}} &&& [[\mathcal A]]_0^{S^{m-2}}}
\end{equation}
The columns on the left hand side and on the right hand side are exact. The diagram, however, does not commute, so that the map $s_m$, as a restriction of $(\lambda^{\ast})^{-1}s^B_{m-1}$ is not defined directly.

\begin{claim}\label{gammahom}
For $x \in B_{m-1}$ there is a \emph{$\Gamma$--homotopy}
\begin{equation*}
\Gamma_x: s_{m-2}d_{m-1}(x) \Rightarrow \iota^{\ast}(\lambda^{\ast})^{-1} s_{m-1}^B(x)
\end{equation*}
which, for $x \in {\rm ker}(d_{m-1})$, yields a homotopy
\begin{equation*}
\Gamma_x: 0 \Rightarrow \iota^{\ast}(\lambda^{\ast})^{-1} s_{m-1}^B(x),
\end{equation*}
so that $s_{m-1}(x) = (\lambda^{\ast})^{-1} s_{m-1}^B(x) \square \Gamma_X$ is defined. Here $s_{m-1}$ is a homomorphism of $\mathbb G$--modules.
\end{claim}

Using Claim \ref{gammahom}, we define $s_{m-1}$ in (\ref{extdiag}) and hence we define $B_m^{(m+1)}$ by the pull back (\ref{pb}). This defines $B_{\ast}^{(m+1)}$ as a $\mathbb G$--module. In order to define $B_{\ast}^{(m+1)}$ as an $(m+1)$--algebra, we need the multiplication maps
\begin{equation}
\begin{array}{lllr}
B_m^{(m+1)} \otimes B_0 & \longrightarrow & B_m^{(m+1)} &\\
B_0 \otimes B_m^{(m+1)} & \longrightarrow & B_m^{(m+1)}  & \\
B_i \otimes B_j & \longrightarrow & B_m^{(m+1)}  & {\rm for} \ i+ j = m, i < m, j < m.
\end{array}
\end{equation}
An element in $B_m^{(m+1)}$ is a pair $(H, x)$ with $x \in B_{m+1}$ and $H: s_{m-1}(x) \Rightarrow 0$ in the groupoid $[[\mathcal A]]^{S^{m-1}}$. For $y \in B_0$ the products $(H, x) \cdot y$ and $y \cdot (H, x)$ are given by the $\Gamma$--tracks
\begin{equation*}
\begin{array}{llll}
\Gamma(x, y): & s_{m-1}(x) \circ s_0(y) & \Rightarrow & s_m(x \cdot y) \\
\Gamma(y, x): & s_0(y) \circ s_{m-1}(x) & \Rightarrow & s_m(y \cdot x),
\end{array}
\end{equation*}
that is, we set
\begin{equation*}
\begin{array}{lll}
(H, x) \cdot y & = & \big( H \circ s_0(y) \square \Gamma(x, y)^{\rm op}, x \cdot y \big) \\
y \cdot (H, x) & = & \big( s_0(y) \circ H \square \Gamma(y, x)^{\rm op}, y \cdot x \big).
\end{array}
\end{equation*}
Here the product of $H$ and $s_0(y)$ is given by (\ref{comp}).

Moreover, for $x \in B_i, y \in B_j$ with $i + j = m, i < m, j < m$ we obtain the product $x \cdot y = \big( H(x, y) , d(x \cdot y) \big)$ where
\begin{equation*}
d(x \cdot y) = (dx) \cdot y + (-1)^i x \cdot (dy).
\end{equation*}
Here the track $H(x, y): s_{m-1}\big(d(xy)\big) \Rightarrow 0$ is the composite
\begin{equation*}
H(x, y) = \pi(\lambda^{\ast})^{-1} \big( s_i^B(x) \circ s_j^B(y) \big) \square \Gamma(x, y)^{\rm op},
\end{equation*}
where the \emph{$\Gamma$--track} is of the form
\begin{equation*}
\Gamma(x, y): \iota^{\ast}(\lambda^{\ast})^{-1} \big( s_i^B(x) \circ s_j^B(y) \big) \Rightarrow s_{m-1}\big( d(x \cdot y) \big).
\end{equation*}


\begin{thebibliography}{99}

\bibitem{Baues1}
H.J. Baues: \emph{The algebra of secondary cohomology operations}, Progress in Mathematics, 247. Birkhauser Verlag, Basel (2006).

\bibitem{Baues2}
H.J. Baues:  \emph{Combinatorial Homotopy and $4$--Dimensional Complexes}, de Gruyter Expositions in Mathematics, 2. Walter de Gruyer and Co., Berlin (1991).

\bibitem{BauesPirashvili}
H.J. Baues and T. Pirashvili: \emph{Comparison of Mac\ Lane, Shukla and Hochschild cohomologies}, J. Reine Angew. Math. {\bf 598} (2006), 25 -- 69.

\bibitem{BauesMuro}
H.J. Baues and F. Muro: \emph{The algebra of secondary homotopy operations in ring spectra}, preprint MPIM: 137 (2006).

\bibitem{BauesJibladze1}
H.J. Baues and M. Jibladze: \emph{Secondary derived functors and the Adams spectral sequence}, Topology {\bf 45} (2006), 295--324.

\bibitem{BauesJibladze2}
H.J. Baues and M. Jibladze: \emph{Dualization ???}, ???.

\bibitem{HardieKampsKieboom}
K. Hardie, K. Kamps and R. Kieboom: \emph{A homotopy $2$--groupoid of a Hausdorff space}, Appl. Categ. Structures {\bf 8} (2000), 209--234.


\end{thebibliography}
 \end{document}